\documentclass[12pt]{amsart}
\usepackage{graphicx}
\usepackage[cp1250]{inputenc}

\newtheorem{theorem}{Theorem}[section]
\newtheorem{proposition}[theorem]{Proposition}
\newtheorem{example}[theorem]{Example}

\newtheorem{corollary}[theorem]{Corollary}
\newtheorem{lemma}[theorem]{Lemma}

\renewcommand{\Re}{\mathop{\rm Re}}

\begin{document}
\title[Necessary condition for compactness$\ldots$]{Necessary condition for compactness of a difference of composition operators on the Dirichlet space}
\author{Ma{\l}gorzata Michalska, Andrzej M. Michalski}

\address{
Ma{\l}gorzata Michalska,  \newline Institute of Mathematics,
\newline Maria Curie-Sk{\l}odowska University, \newline pl. M.
Curie-Sk{\l}odowskiej 1, \newline 20-031 Lublin, Poland}
\email{malgorzata.michalska@poczta.umcs.lublin.pl}

\address{
Andrzej M. Michalski, \newline
Department of Complex Analysis, \newline
The John Paul II Catholic University of Lublin, \newline
ul. Konstantyn\'{o}w 1H, \newline
20-950 Lublin, Poland}
\email{amichal@kul.lublin.pl}

\date{\today}


\maketitle

\begin{abstract}
Let $\varphi$ be a self-map of the unit disk and let
$C_\varphi$ denote the composition operator acting on the standard
Dirichlet space $\mathcal{D}$. A necessary condition for compactness
of a difference of two bounded composition operators acting on
$\mathcal{D}$, is given. As an application, a characterization of
disk automorphisms $\varphi$ and $\psi$ for which the commutator
$[C^*_\psi,C_\varphi]$ is compact, is given.
\end{abstract}

\baselineskip1.4\baselineskip

\section{Introduction}
Let $\mathbb{D}=\{z:\,|z|<1\}$ denote the open unit disk in the
complex plane $\mathbb{C}$ and let $\mathbb{T}=\{z:\,|z|=1\}$ denote
the unit circle in $\mathbb{C}$. The Dirichlet space $\mathcal{D}$
is the space of all analytic functions $f$ in $\mathbb{D}$, such
that
  \begin{equation*}%\label{norm_dirich}
    \|f\|^2_\mathcal{D}:=|f(0)|^2+ \int_{\mathbb{D}}|f^\prime(z)|^2dA(z)<\infty,
  \end{equation*}
where $dA(z)=\pi^{-1}dxdy$ is the normalized two dimensional
Lebesgue measure on $\mathbb{D}$. The Dirichlet space is a Hilbert
space with inner product
\begin{equation*}%\label{inner_product}
  \langle f,g\rangle_\mathcal{D}:=f(0)\overline{g(0)}+\int_{\mathbb{D}} f^\prime(z)\overline{g^\prime(z)}dA(z).
\end{equation*}
The Dirichlet space has the reproducing
kernel property and the kernel function
is defined as
\begin{equation}\label{kernel}
  K_w(z):=1+\log\frac{1}{1-\overline{w}z},
\end{equation}
where the branch of the logarithm is chosen such that
\begin{equation*}%\label{kernel_log}
  \overline{K_w(z)}=\overline{\langle K_w,K_z\rangle_\mathcal{D}}=\langle K_z,K_w\rangle_\mathcal{D}=K_z(w).
\end{equation*}

By a self-map of $\mathbb{D}$ we mean an analytic function $\varphi$
such that $\varphi(\mathbb{D})\subset\mathbb{D}$. We will also
assume that a self-map $\varphi$ is not a constant function. For a self-map of the unit disk $\varphi$, the composition operator
$C_{\varphi}$ on the Dirichlet space $\mathcal{D}$ is defined by
$C_{\varphi}f:=f\circ \varphi$. The composition operator $C_{\varphi}$ on Dirichlet space
is not necessarily bounded for an arbitrary self-map of the
unit disk. However, $C_{\varphi}$ is bounded on
$\mathcal{D}$ if, for example, $\varphi$ is a finitely valent
function (see, e.g., \cite{MV2,Z2}). More is known about the composition
operator $C_{\varphi}$ when the symbol $\varphi$ is a
linear-fractional self-map of the unit disk of the form
  $$\varphi(z):=\frac{az+b}{cz+d},$$
where $ad-bc\not= 0$.
In that case $C_\varphi$ is compact on $\mathcal{D}$
if and only if $\|\varphi\|_\infty<1$ (see, e.g., \cite{CM,S,Z2}).

For an arbitrary self-map of the
unit disk $\varphi$, if the operator $C_{\varphi}$ is bounded, then the adjoint operator
$C_{\varphi}^{*}$ satisfies
\begin{equation*}%\label{adjoint_definition}
  C_{\varphi}^{*}f(w)=\langle f,K_{w}\circ \varphi\rangle_\mathcal{D},
\end{equation*}
which yields useful equality
\begin{equation}\label{adjoint_kernel_property}
  C^*_\varphi K_w=K_{\varphi(w)}.
\end{equation}
For $\varphi$ a
linear-fractional self-map of $\mathbb{D}$, Gallardo-Guti\'{e}rrez and Montes-Rodr\'{i}guez in
\cite{GM} (see also \cite{MV}) proved that the adjoint of the
composition operator is given by formula
\begin{equation}\label{Martin_Vukotic_adjoint}
  C_{\varphi}^{*}f=f(0)K_{\varphi(0)}-s(C_{\varphi^{*}}f)(0)+sC_{\varphi^{*}}f,
\end{equation}
where $s:=ad-bc$ and
  $$\varphi^{*}(z):=\frac{1}{\overline{\varphi^{-1}(\frac{1}{\overline{z}})}}, \quad z\in\mathbb{D}$$
is the Krein adjoint of $\varphi$. It is worth to note that $\varphi^{*}$ is a linear-fractional self-map of the unit disk, in fact
  $$\varphi^{*}(z)=\frac{\overline{a}z-\overline{c}}{-\overline{b}z+\overline{d}}.$$
It is easy to check that $w$ is a fixed point of $\varphi$ if and
only if $1/\overline{w}$ is a fixed point of $\varphi^*$. In
particular, if $\varphi$ has a fixed point on $\mathbb{T}$ then it
is a fixed point of both $\varphi$ and $\varphi^*$.

Let $\varphi$ be a disk automorphism, which is of the form
\begin{equation}\label{automorph}
  \varphi(z)=e^{i\theta}\frac{a-z}{1-\overline{a}z},\quad z\in
  \mathbb{D},
\end{equation}
where $a\in \mathbb{D}$ and $\theta\in(-\pi,\pi]$. We will say that
\begin{itemize}
  \item $\varphi$ is elliptic if and only if $|a|<\cos\frac{\theta}{2}$,
  \item $\varphi$ is parabolic if and only if $|a|=\cos\frac{\theta}{2}$,
  \item $\varphi$ is hyperbolic if and only if $|a|>\cos\frac{\theta}{2}$,
\end{itemize}
(see,
e.g., \cite[Ex. 4, p. 7]{S}). One can easily verify that if $\varphi$ is elliptic then $\varphi^*$ is also elliptic.

For $\varphi$ and $\psi$, two linear-fractional self-maps of
$\mathbb{D}$, we consider the commutator
  $$[C_{\psi}^{*},C_{\varphi}]:=C_{\psi}^{*}C_{\varphi}-C_{\varphi}C_{\psi}^{*}$$
on $\mathcal{D}$. The compactness of the commutator can be expressed
by setting conditions on the maps $\psi$ and $\varphi$. The
commutator $[C_{\psi}^{*},C_{\varphi}]$ is trivially compact on
$\mathcal{D}$ if it is equal to zero, or when
$C_{\psi}^{*}C_{\varphi}$ and $C_{\varphi}C_{\psi}^{*}$ are both
compact. In particular, this happens when $\|\psi\|_{\infty}<1$ or
$\|\varphi\|_{\infty}<1$. Thus, to avoid triviality, we will
consider only composition operators, and their adjoints, whose
symbols are the linear-fractional self-maps of $\mathbb{D}$ with
$\|\psi\|_{\infty}=\|\varphi\|_{\infty}=1$.

We should mention, that if $\varphi$ and $\psi$ are two
linear-fractional self-maps of $\mathbb{D}$ then there are known
conditions for non-trivial compactness of the commutator
$[C_{\psi}^{*},C_{\varphi}]$ acting on the Hardy space
$H^2$ obtained by Clifford et al. \cite{CLN}, and
acting on the weighted Bergman spaces $A_{\alpha}^2(\mathbb{D})$
obtained by MacCluer et al. \cite{MNW}. Their results were obtained for $\|\psi\|_{\infty}=\|\varphi\|_{\infty}=1$ in the case
when both $\varphi$ and $\psi$ are disk
automorphisms, and in the case when at least one of the maps is not an
automorphism. In particular, they proved that in the first case the commutator is non-trivially
compact if and only if both maps are rotations. We refer the reader to \cite{MNW} for more background information.

In this paper we study properties of the difference of two composition operators defined on the
Dirichlet space. In Section 2 we give a necessary condition for
compactness of the difference of two bounded
composition operators. In Section 3, as an application of our necessary condition for compactness, we determine when the commutator $[C_{\psi}^{*},C_{\varphi}]$, with both symbols
$\varphi$ and $\psi$ being disk automorphisms and not equal to the
identity, is compact.

\section{Difference of two composition operators}
\setcounter{equation}{0}

To study compactness of the commutator $[C_{\psi}^{*},C_{\varphi}]$
we need to know when a difference of two composition operators is
compact. There are known conditions for compactness of a difference
of composition operators for weighted Dirichlet spaces obtained
by Moorhouse in \cite{M}. Unfortunately, these results do not apply
to the classical Dirichlet space $\mathcal{D}$. In Theorem
\ref{compact_diff} we give a necessary condition for compactness of
the difference of two bounded composition operators on $\mathcal{D}$. First,
we prove a technical lemma.

\begin{lemma}\label{log_lim}
Let the sequences $\mathbb{N}\ni n\mapsto a_n \in(0,1)$ and $\mathbb{N}\ni
n\mapsto b_n \in(0,1)$ converge to $0$ and let
$\lim_{n\to\infty}b_n/a_n=0$. Then there exists a positive integer $N$
such that
  $$0<\frac{\ln a_n}{\ln b_n}<1,$$
for all $n>N$.
\end{lemma}
\begin{proof}
Let the sequences $\mathbb{N}\ni n\mapsto a_n \in(0,1)$ and $\mathbb{N}\ni
n\mapsto b_n \in(0,1)$ be as required in our lemma. By the
assumption $\lim_{n\to\infty}b_n/a_n=0$ we know that there
exists positive integer $N$ such that $b_n/a_n<1$ for all $n>N$.
Thus, for $n>N$ we have $\ln\frac{b_n}{a_n}<0$ and since both $\ln a_n$ and $\ln b_n$ are negative
  $$0<\frac{\ln a_n}{\ln b_n}<1.$$
This completes
the proof.
\end{proof}

Now, we are ready to state our main result.
\begin{theorem}
\label{compact_diff}
Let $\varphi$ and $\psi$ be self-maps of the unit disk $\mathbb{D}$ such that the composition operators $C_\varphi,C_\psi$ induced by $\varphi$ and $\psi$, respectively, are bounded.
If $C_\varphi -C_\psi$ is compact on $\mathcal{D}$ then
\begin{equation}\label{compact_difference}
  \lim_{|w|\to 1^-}\left\{\frac{1-|w|^2}{1-|\varphi(w)|^2}+\frac{1-|w|^2}{1-|\psi(w)|^2}\right\}|\varphi(w)-\psi(w)|=0.
\end{equation}
\end{theorem}

\begin{proof}
Clearly, $C_\varphi -C_\psi$ is compact if and only if
$C^*_\varphi-C^*_\psi$ is compact. Therefore, it is enough to prove
that if \eqref{compact_difference} does not hold, then the operator
$C^*_\varphi-C^*_\psi$ is not compact on $\mathcal{D}$. Assume that
the limit in \eqref{compact_difference} does not exist or it exists,
but it is not equal to $0$. In both cases one can find a sequence
$\mathbb{N}\ni n\mapsto w_n \in\mathbb{D}\setminus\{0\}$, with
$|w_n|\to 1^-$, such that the following conditions are satisfied:
\begin{enumerate}
  \item[(i)] $\displaystyle\lim_{n\to \infty}\left\{\frac{1-|w_n|^2}{1-|\varphi(w_n)|^2}+\frac{1-|w_n|^2}{1-|\psi(w_n)|^2}\right\}|\varphi(w_n)-\psi(w_n)|\not=0,$
  \item[(ii)] the limits $\displaystyle\psi_0:=\lim_{n\to\infty}\psi(w_n)$ and $\displaystyle\varphi_0:=\lim_{n\to\infty}\varphi(w_n)$ exist,
  \item[(iii)] the limits $\displaystyle\Phi_0:=\lim_{n\to \infty}\frac{1-|w_n|^2}{1-|\varphi(w_n)|^2}$ and $\displaystyle\Psi_0:=\lim_{n\to \infty}\frac{1-|w_n|^2}{1-|\psi(w_n)|^2}$ exist.
%$\psi(w_n)\not=0$ and $\varphi(w_n)\not=0$ for all $n\in\mathbb{N}$,\item[(iv)]
\end{enumerate}
Indeed, such a sequence exists. Observe, that if $\varphi$ is a
self-map of the unit disk, then as a consequence of Schwarz-Pick
lemma we have (see, e.g., \cite[Corollary 2.40]{CM})
\begin{equation}\label{analytic_upperbound}
  \frac{1-|\varphi(w_n)|}{1-|w_n|}\geq \frac{1-|\varphi(0)|}{1+|\varphi(0)|}, \quad
  w\in\mathbb{D}
\end{equation}
and thus both factors in the limit in condition (i) are bounded.
Consequently, by passing to a subsequence if necessary, we obtain sequence
satisfying (i)-(iii).

Now, we consider a sequence of normalized kernel functions
$K_{w_n}/\|K_{w_n}\|$, where $K_{w_n}$ is given by \eqref{kernel},
and we show that $\|(C^*_\varphi-C^*_\psi)K_{w_n}\|/\|K_{w_n}\|$
does not tend to $0$. Since $K_{w_n}/\|K_{w_n}\|\to 0$ weakly, this
disproves that $C^*_\varphi-C^*_\psi$ is compact (see, e.g.,
\cite[Theorem 1.3.4]{Zhu}).

Note, that $\|K_{w_n}\|^2=1+\log(1/(1-|{w_n}|^2))$ which together with \eqref{adjoint_kernel_property} yields
\begin{eqnarray}\label{adjoint_diff}
  \frac{\left\|(C^*_\varphi-C^*_\psi)K_{w_n}\right\|^2}{\|K_{w_n}\|^2}
  &=& \frac{\|K_{\varphi({w_n})}\|^2+\|K_{\psi({w_n})}\|^2-2\Re\langle K_{\varphi({w_n})},K_{\psi({w_n})}\rangle}{\|K_{w_n}\|^2}\nonumber\\
  &=& \frac{\ln\frac{1}{1-|\varphi({w_n})|^2} +\ln\frac{1}{1-|\psi({w_n})|^2} -2\ln\frac{1}{|1-\overline{\varphi({w_n})}\psi({w_n})|}}{1+\ln\frac{1}{1-|{w_n}|^2}}.
\end{eqnarray}

Observe, that by \eqref{analytic_upperbound}, none of the factors in the
limit in condition (i) can tend to $0$. Thus, in particular, $\varphi_0\not=\psi_0$ and
\begin{equation}\label{low_bound}
  0<\frac{|1-\overline{\varphi_0}\psi_0|}{2}<|1-\overline{\varphi(w_n)}\psi(w_n)|\leq 2,
\end{equation}
for sufficiently large $n$.

It is enough to consider three
cases:
\begin{eqnarray*}
  &\text{\bf Case I:}& |\varphi_0|=1 \qquad\text{and}\qquad |\psi_0|<1,\ \text{or}\\
  &\text{\bf Case II:}& |\varphi_0|<1 \qquad\text{and}\qquad |\psi_0|=1, \ \text{or}\\
  &\text{\bf Case III:}& |\psi_0|=|\varphi_0|=1 \qquad\text{and}\qquad \psi_0\not=\varphi_0.
\end{eqnarray*}

{\bf Case I.} Let $|\varphi_0|=1$ and $|\psi_0|<1$. Then $\Psi_0=0$ and $\Phi_0>0$, by (i).  So, for a sufficiently large $n$, say $n>N$, we have
  $$0<\frac{1-|\psi_0|^2}{2}<1-|\psi(w_n)|^2< 1,$$
and
\begin{equation}\label{phi_bound}
  0<\frac{\Phi_0}{2}<\frac{1-|w_n|^2}{1-|\varphi(w_n)|^2}<2\frac{1+|\varphi(0)|}{1-|\varphi(0)|},
\end{equation}
where the last inequality follows from \eqref{analytic_upperbound}.
Hence,
 $$\kappa(w_n):=\ln\frac{1}{1-|\psi(w_n)|^2}
  -2\ln\frac{1}{|1-\overline{\varphi(w_n)}\psi(w_n)|}+\ln\frac{1-|w_n|^2}{1-|\varphi(w_n)|^2}$$
is bounded and by \eqref{adjoint_diff}, we get
\begin{eqnarray*}
  \lim_{n\to\infty}\frac{\left\|(C^*_\varphi-C^*_\psi)K_{w_n}\right\|^2}{\|K_{w_n}\|^2}
  &=& \lim_{n\to\infty}\frac{\ln\frac{1}{1-|w_n|^2} +\kappa(w_n)}{1+\ln\frac{1}{1-|w_n|^2}}=1.
\end{eqnarray*}

{\bf Case II.} If $|\varphi_0|<1$ and $|\psi_0|=1$, then the proof of the following equality
  $$\lim_{n\to\infty}\frac{\left\|(C^*_\varphi-C^*_\psi)K_{w_n}\right\|^2}{\|K_{w_n}\|^2}=1$$
proceeds analogously to the proof in Case I.

{\bf Case III.} Let $|\varphi_0|=|\psi_0|=1$ and
$\varphi_0\not=\psi_0$. Note, that $\Phi_0$ and $\Psi_0$ can not both be equal to $0$.

If $\Phi_0=0$ and $\Psi_0\not=0$, then there exists positive integer $N$, such that for $n>N$
\begin{equation}\label{psi_bound}
  0<\frac{\Psi_0}{2}<\frac{1-|w_n|^2}{1-|\psi(w_n)|^2}<2\frac{1+|\psi(0)|}{1-|\psi(0)|}
\end{equation}
and \eqref{low_bound} hold. Moreover, since $|\varphi_0|=|\psi_0|=1$, we may assume, by passing to a subsequence if necessary, that $\varphi(w_n)\not=0$ and $w_n\not=0$ for each $n$. Now, we can use Lemma
\ref{log_lim} with $a_n=1-|\varphi(w_n)|^2$ and $b_n=1-|w_n|^2$ and get that there exists positive integer $N_1>N$ such that
  $$0<\frac{\ln(1-|\varphi(w_n)|^2)}{\ln(1-|w_n|^2)}=\frac{\ln\frac{1}{1-|\varphi(w_n)|^2}}{\ln\frac{1}{1-|w_n|^2}}<1,$$
for all $n>N_1$. By passing to a subsequence, if necessary, we can assume that the limit $\lim_{n\to\infty}\ln(1-|\varphi(w_n)|^2)/\ln(1-|w_n|^2)$ exists.
Hence,
  $$\lambda(w_n):= -2\ln\frac{1}{|1-\overline{\varphi(w_n)}\psi(w_n)|}+\ln\frac{1-|w_n|^2}{1-|\psi(w_n)|^2}$$
is bounded and, by \eqref{adjoint_diff}, we get
\begin{eqnarray*}
  \lim_{n\to\infty}\frac{\left\|(C^*_\varphi-C^*_\psi)K_{w_n}\right\|^2}{\|K_{w_n}\|^2} &=& \lim_{n\to\infty}
  \frac{
  \ln\frac{1}{1-|\varphi(w_n)|^2}+\ln\frac{1}{1-|w_n|^2} +\lambda(w_n)}
  {1+\ln\frac{1}{1-|w_n|^2}}\geq 1.
\end{eqnarray*}

If $\Phi_0\not=0$ and $\Psi_0=0$ then \eqref{low_bound} and \eqref{phi_bound} hold for n sufficiently large and
  $$\widetilde{\lambda}(w_n):= -2\ln\frac{1}{|1-\overline{\varphi(w_n)}\psi(w_n)|}+\ln\frac{1-|w_n|^2}{1-|\varphi(w_n)|^2}$$
is bounded. Another application of Lemma \ref{log_lim} ensures that
%by \eqref{adjoint_diff}
\begin{eqnarray*}
  \lim_{n\to\infty}\frac{\left\|(C^*_\varphi-C^*_\psi)K_{w_n}\right\|^2}{\|K_{w_n}\|^2} &=& \lim_{n\to\infty}
  \frac{
  \ln\frac{1}{1-|\psi(w_n)|^2}+\ln\frac{1}{1-|w_n|^2}+\widetilde{\lambda}(w_n)}
  {1+\ln\frac{1}{1-|w_n|^2}}\geq 1.
\end{eqnarray*}

Finally, if $\Phi_0\not=0$ and $\Psi_0\not=0$ then \eqref{low_bound}, \eqref{phi_bound} and  \eqref{psi_bound} hold for sufficiently large $n$. Hence,
  $$\widehat{\lambda}(w_n):= -2\ln\frac{1}{|1-\overline{\varphi(w_n)}\psi(w_n)|}+\ln\frac{1-|w_n|^2}{1-|\varphi(w_n)|^2}+\ln\frac{1-|w_n|^2}{1-|\psi(w_n)|^2}$$
is bounded and by \eqref{adjoint_diff} we get
\begin{eqnarray*}
  \lim_{n\to\infty}\frac{\left\|(C^*_\varphi-C^*_\psi)K_{w_n}\right\|^2}{\|K_{w_n}\|^2} &=& \lim_{n\to\infty}
  \frac{
  2\ln\frac{1}{1-|w_n|^2}+\widehat{\lambda}(w_n)}
  {1+\ln\frac{1}{1-|w_n|^2}}=2.
\end{eqnarray*}
This completes the proof.
\end{proof}

The above theorem is in particular true for all finitely valent self-maps of the unit disk. Moreover, in the case of disk
automorphisms of the form \eqref{automorph} we can obtain much more simple condition.

\begin{corollary}
\label{compact_diff_automor}
Let $\varphi$ and $\psi$ be disk automorphisms given by \eqref{automorph}. If $C_\varphi -C_\psi$ is compact on $\mathcal{D}$ then $\varphi=\psi$.
\end{corollary}
\begin{proof}
Let $\varphi$ and $\psi$ be disk automorphisms given by \eqref{automorph} and assume that
$C_\varphi -C_\psi$ is compact on $\mathcal{D}$. We show that
$\varphi(\zeta)=\psi(\zeta)$ for all $\zeta\in\mathbb{T}$.

Fix $\zeta\in\mathbb{T}$. By Theorem
\ref{compact_diff} we know that the compactness of the difference
$C_\varphi -C_\psi$ implies
\begin{equation}\label{thm_week}
  \lim_{z\to
  \zeta}\left\{\frac{1-|z|^2}{1-|\varphi(z)|^2}+\frac{1-|z|^2}{1-|\psi(z)|^2}\right\}|\varphi(z)-\psi(z)|=0.
\end{equation}
We show that neither $(1-|z|^2)(1-|\varphi(z)|^2)^{-1}$ nor $(1-|z|^2)(1-|\psi(z)|^2)^{-1}$
can tend to $0$ as $z$ tends to $\zeta$. Indeed, for $\varphi$ given by \eqref{automorph}, we
have
\begin{eqnarray*}
  \frac{1-|z|^2}{1-|\varphi(z)|^2} &=&
%\frac{1-|z|^2}{1-|e^{i\theta}\frac{a-z}{1-\overline{a}z}|^2}
%=\frac{(1-|z|^2)|{1-\overline{a}z}|^2}{1+|\overline{a}z|^2-|a|^2-|z|^2}=
  \frac{|{1-\overline{a}z}|^2}{1-|a|^2}\geq\frac{1-|a|}{1+|a|}>0,
\end{eqnarray*}
for all $z\in\mathbb{\overline{D}}$. Hence, $(1-|z|^2)(1-|\varphi(z)|^2)^{-1}$ can not tend to
$0$ as $z \to \zeta$. The same argument can be used
to show that $(1-|z|^2)(1-|\psi(z)|^2)^{-1}$ does not tend to
$0$ as $z\to \zeta$. Thus, \eqref{thm_week} implies
  $$\lim_{z\to\zeta}|\varphi(z)-\psi(z)|=0,$$
and $\varphi(\zeta)=\psi(\zeta)$. Since $\zeta$ was chosen arbitrarily, our claim follows.
\end{proof}

\section{Commutator}
\setcounter{equation}{0}

In this section we study some properties of the commutator
$[C^*_{\psi},C_\varphi]$ with $\varphi$ and $\psi$ being disk automorphisms.

For $f,g\in\mathcal{D}$ one can define the following rank-one
operator
  $$f\otimes g(h):=\langle h,g\rangle_{\mathcal{D}}f, \quad h\in\mathcal{D}.$$
By \eqref{Martin_Vukotic_adjoint}, for an arbitrary
linear-fractional self-map $\psi$, the adjoint of the composition
operator $C_{\psi}$ can be written as
\begin{equation}\label{adjoint_compact}
  C^*_{\psi}=sC_{\psi^*}+K,
\end{equation}
where $Kf:=(K_{\psi(0)}\otimes K_0)(f)-s (K_0\otimes
K_0)(C_{\psi^*}f)$, $K_w$ is a kernel function given by
\eqref{kernel} and $s$ is defined in \eqref{Martin_Vukotic_adjoint}.
Obviously, K is a compact operator on $\mathcal{D}$. Hence, we have
\begin{eqnarray}\label{commutator_diff_compact}
  [C^*_{\psi},C_\varphi]f&=&C^*_{\psi}C_\varphi f-C_\varphi C^*_{\psi}f
  %=s(C_{\psi^*}C_\varphi -C_\varphi C_{\psi^*})f+Lf\nonumber\\
  =s(C_{\psi^*\circ \varphi} -C_{\varphi\circ \psi^*})f+Lf,
\end{eqnarray}
where
\begin{equation}\label{commut_diff}
L:=KC_\varphi-C_\varphi K=[K,C_\varphi]
\end{equation}
is again compact and $L\not=0$, unless both $\varphi$ and $\psi$ are equal to the identity.

\begin{theorem}
\label{commut_compact_automor} Let $\varphi$, $\psi$ be
disk automorphisms given by \eqref{automorph}, none of which is the
identity. Then the commutator $[C^*_{\psi},C_\varphi]$ is
non-trivially compact if and only if either both $\varphi$ and $\psi^*$
have the same set of fixed points, or both $\varphi$ and $\psi$
are elliptic.
\end{theorem}
\begin{proof}
Assume first, that either both $\varphi$ and $\psi^*$ have the same
set of fixed points, or both $\varphi$ and $\psi$ are elliptic.
Then by \cite[Theorem 2, p. 72]{L} we know that $\varphi$ and
$\psi^*$ commute, that is $\psi^*\circ \varphi=\varphi\circ \psi^*$. Thus, the difference $C_{\psi^*\circ \varphi}
-C_{\varphi\circ \psi^*}$ in \eqref{commutator_diff_compact} is
equal to zero and the commutator $[C^*_{\psi},C_\varphi]$ is
non-trivially compact, since $\varphi$ and $\psi$ are not equal to the identity and $L\not=0$ (see formula \eqref{commut_diff}).

Now, assume that the commutator $[C^*_{\psi},C_\varphi]$ is
non-trivially compact. Then, by formula
\eqref{commutator_diff_compact} $C_{\psi^*\circ \varphi}
-C_{\varphi\circ \psi^*}$ is also compact and Corollary
\ref{compact_diff_automor} implies that $\psi^*\circ \varphi =\varphi\circ
\psi^*$. Finally, again by \cite[Theorem 2, p. 72]{L} we obtain that
either both $\varphi$ and $\psi^*$ have the same set of fixed points,
or both $\varphi$ and $\psi$ are elliptic.
\end{proof}
We say that the composition
operator $C_\varphi$ is essentially normal if the self-commutator $[C^*_{\varphi},C_\varphi]$ is compact.
This property was studied in \cite{BLNS,Z} for composition operator defined on the Hardy space and in \cite{MW} for composition operators defined on the weighted Bergman spaces.
As a consequence of Theorem \ref{commut_compact_automor} we get the following sufficient condition for $C_\varphi$ to be essentially normal on the Dirichlet space $\mathcal{D}$.

\begin{corollary}
If $\varphi$ is a disk automorphism given by \eqref{automorph}, then the composition operator $C_\varphi$ is
essentially normal.
\end{corollary}
\begin{proof}
Let $\varphi$ be given by \eqref{automorph}. If $\varphi$ is equal
to the identity then $[C^*_{\varphi},C_\varphi]=0$. Now assume that
$\varphi$ is not the identity map. If $\varphi$ is elliptic, then by
Theorem \ref{commut_compact_automor} the commutator
$[C^*_{\varphi},C_\varphi]$ is non-trivially compact. We show that
if $\varphi$ is not elliptic, then $\varphi$ and $\varphi^*$ have
the same set of fixed points. This follows from our observation that
$\varphi$ and $\varphi^*$ have the same set of fixed points on
$\mathbb{T}$ (see, Section 1. Introduction). Indeed, if $\varphi$ is
a parabolic automorphism, then it has only one fixed point
$z=(1+e^{i\theta})/(2\overline{a})\in \mathbb{T}$, and if $\varphi$
is a hyperbolic automorphism, then it has two fixed points
$z_{k}=(1+e^{i\theta/2})\overline{a}^{-1}(\cos\theta/2 +(-1)^k
i\sqrt{|a|^2-\cos^2\theta/2})\in \mathbb{T}$, $k=1,2$. Again, by
Theorem \ref{commut_compact_automor}, the commutator
$[C^*_{\varphi},C_\varphi]$ is non-trivially compact, which
completes the proof.
\end{proof}

\end{document}